\documentclass[a4paper,english,oneside]{amsart}
\usepackage[T1]{fontenc}
\usepackage[hscale=0.6, vscale=0.66]{geometry}

\usepackage{hyperref}

\usepackage[numbers,sort&compress]{natbib}

\usepackage{amssymb}

\DeclareFontFamily{OMX}{mlmex}{}
\DeclareFontShape{OMX}{mlmex}{m}{n}{%
   <->mlmex10%
   }{}%
\usepackage{mlmodern}

\theoremstyle{plain}
\newtheorem{theorem}{Theorem}
\newtheorem{proposition}{Proposition}
\newtheorem{lemma}{Lemma}
\newtheorem{corollary}{Corollary}

\theoremstyle{definition}
\newtheorem{definition}{Definition}
\newtheorem{remark}{Remark}

\usepackage{color}

\usepackage[np,autolanguage]{numprint}

\newcommand\cX{\mathcal{X}}
\newcommand\cA{\mathcal{A}}

\newcommand\NN{\mathbb{N}}
\newcommand\ZZ{\mathbb{Z}}
\newcommand\RR{\mathbb{R}}
\newcommand\sD{\mathsf{D}}
\newcommand\sA{\mathsf{A}}
\newcommand\ld{\mathrm{ld}}

\begin{document}

\mbox{}
\raise2\baselineskip\hbox{\smash{\parbox{\linewidth}{This preprint was published (with some minor differences) as
\emph{Measures for the summation of {I}rwin series}.
Integers (2026), Paper No. A11.\newline \url{https://math.colgate.edu/~integers/aa11/aa11.pdf}
\endgraf
A CORRIGENDUM TO THE PUBLISHED PAPER (hence also to the Lemma 3 in the text here) WILL BE FOUND AT THE END OF THE BODY.}}}

\title[Measures for Irwin series]{Measures for the summation of Irwin series}

\author[J.-F. Burnol]{Jean-François Burnol}

\address{Université de Lille, Faculté des Sciences et technologies,
  Département de mathématiques, Cité Scientifique, F-59655 Villeneuve d'Ascq
  cedex, France}
\email{jean-francois.burnol@univ-lille.fr}
\date{December 21, 2025. To appear in Integers (2026). This ultimate
  pre-publication version updates the numerotation of lemmas, propositions,
  and theorems to match the one of the upcoming publication.}

\begin{abstract}
We consider the series of reciprocals of those positive integers with exactly
$k$ occurrences of a given $b$-ary digit $d$ (Irwin series), and obtain
geometrically convergent representations for their sums.  They are expressed
in terms of the moments and Stieltjes transforms of some measures on the unit
interval.  The moments obey linear recurrences allowing straightforward
numerical implementations.  This framework also allows to study the large $k$
limit, and, in further works by the author, the large $b$ asymptotics.
\end{abstract}

\maketitle
\section{Introduction}

The Irwin series \cite{irwin} are sub-series of the harmonic series with
conditions on the number of occurrences of digits in the denominators.  For
example, only those positive integers having at most 77 occurrences of the
digit $3$, and/or at most 125 occurrences of the digit $5$, are
kept.  Irwin showed that such series converge.
This generalized an earlier contribution by Kempner
\cite{kempner} where the convergence was shown for those series whose terms
have no occurrence of a given digit.

Hardy and Wright establish the Kempner result as
Theorem 144 of \cite{hardywright}.  Their statement ``The sum of the
  reciprocals of the numbers which miss a given digit is convergent'' is proven
for an arbitrary radix $r$, whereas Kempner had considered only decimal
figures.  It is found in the section ``The representations of numbers as
decimals'' which discusses topics such as sets of measure zero and the
theorem of Borel stating that almost every real number is normal in any base
\cite[Thm. 148]{hardywright}.

The Kempner and Irwin series have been the object of various studies and
generalizations \cite{allouchecohenetalActa1987, alloucheshallit1989,
  alloucheshallitCUP2003, alloucheshallitsondow2007, allouchemorin2023,
  allouchehumorin2024, aloui2017, baillie1979, baillie2008, farhi, fischer,
  gordon2019, huyining2016, irwin, kempner, klove1971, kohlerspilker2009,
  lubeckponomarenko2018, mukherjeesarkar2021, nathanson2021integers,
  nathanson2022ramanujan, schmelzerbaillie, segalleppfine1970, wadhwa1978,
  walker2020}.  A characteristic of Theorems \ref{thm:main} and
\ref{thm:mainpositive} established in the present work is that they provide theoretically
\emph{exact} formulas which can be converted straightforwardly
into efficient
numerical algorithms%
\footnote{See
  \url{https://burnolmath.gitlab.io/irwin/\#sagemath-implementation}.}.
They are
thus an alternative to the (also efficient!) numerical algorithms of
Baillie \cite{baillie1979,baillie2008}.  In all examples computed by the
author, the results matched those returned by the Baillie algorithms for
Kempner and Irwin series.  In a later work \cite{burnolone42}, we started the
extension to counting occurrences of multi-digit strings $d_1\dots d_p$. Here
too, the numerical results in the case of no occurrence matched what is
provided by the Schmelzer--Baillie \cite{schmelzerbaillie} algorithm. For $p=2$
and one occurrence, the approach of \cite{burnolone42} is currently the
only published algorithm available.

We shall consider here, for any triple of integers $(b,d,k)$ with $b>1$,
$0\leq d<b$, and $k\geq0$, the associated Irwin series, i.e., the subsum of
the harmonic series where only those terms are kept whose denominators have
exactly $k$ occurrences of the $b$-ary digit $d$.  The case $b=2$, $d=1$,
$k=0$ is special: it gives an empty series, and the value of the sum is thus
zero.
We explain how to generalize our earlier work \cite{burnolkempner} which
handled $k=0$.  Here as in \cite{burnolkempner}, the core objects are moments
and Stieltjes transforms of suitable measures on the unit interval.

For $b=10$, Farhi \cite{farhi} proved that the limit as $k\to\infty$ is
$10\log(10)$.  We prove that, in general, it equals $b\log(b)$.  Farhi's
method could be extended to general $b>1$, but our approach is very different.
We actually give two proofs, one a corollary of the convergence of
certain measures toward the Lebesgue measure.  We extend this latter method of
proof in \cite{burnolblocks} to the case of counting occurrences of a
multi-digit string $d_1\dots d_p$, recovering the recent theorem by Allouche,
Hu, and Morin \cite{allouchehumorin2024}, which says that the limit as
$k\to\infty$ in that case is $b^p\log(b)$.

In further works, which were also based upon the present one, we obtained
various new results on Kempner--Irwin series and their generalizations, in
particular regarding their large $b$ asymptotics \cite{burnoldigamma,
  burnollargeb,burnolasymptotic}.

\section{Notation and Terminology}

Let $b>1$ be an integer, which is kept fixed throughout the paper.  The set
$\NN$ is defined as $\ZZ_{\geq0}$ (i.e., it contains zero).  We define
\emph{$b$-imal} numbers as the elements of $\cup_{l\geq0} b^{-l}\ZZ$. For
such a $b$-imal number $x$, the
smallest $l\geq0$ with $x\in b^{-l}\ZZ$ is called its \emph{depth}.

Let $\sD = \{0,\dots,b-1\}$ be the set of $b$-ary digits.  We pick a
$d\in\sD$ which will be fixed throughout the paper.  The case $d=0$ has some
specific features and will require special consideration at some places.  We let
$\sA = \sD\setminus\{d\}$.  We let $N = \# \sA = b-1$ and $N_1 =
\#(\sA\setminus\{0\})$.

The space of \emph{strings} $\boldsymbol{\sD}$ is defined to be the union of all
Cartesian products $\sD^l$ for $l\geq0$.  Notice that for $l=0$ one has
$\sD^0=\{\emptyset\}$, i.e., the set with one element.  We will call this
special element the \emph{none-string}.  The length $|X|=l$ of a string $X$ is
the integer $l$ such that $X\in\sD^l$.  For $l=1$ a string is thus the same
thing as a
digit, although for the sake of clarity it is better to treat
them as separate things.  There is a map from strings to integers which to
$X=(d_{l},\dots ,d_1)$ assigns $n(X) = d_{l} b^{l} + \dots + d_1$, and
sends the none-string to zero.  This map is many-to-one, and each
integer $n$ has a unique minimal-length representation $X(n)$ which is called
the (minimal) \emph{$b$-representation} of $n$.

The length of an integer $n$ is denoted $l(n) = |X(n)|$. Thus, $l(n)$ is the
smallest non-negative integer such that $n<b^{l(n)}$.  Note that $l(0)=0$, not
$1$.  The letter $l$ is also used to refer to a specific integer, e.g., we
may say ``let $l$ be an integer, and consider the integers $m$ such
  that $l(m)=l$''.  We hope that this notation will not create confusion.

Define, for $k\geq 0$:
\[
\cX^{(k)} = \{ X \in \boldsymbol{\sD}\mid X \textrm{ contains
  the digit }d\textrm{ exactly }k\textrm{ times}\}.
\]
The set $\cX^{(k)}$ contains the none-string if and only if $k=0$.  We write
$\cX^{(k)}_l$ for its subsets of strings having length $l$, for $l\geq0$. Let
$\cA^{(k)}$ be the set of non-negative integers with minimal $b$-imal
representation in $\cX^{(k)}$, and $\cA^{(k)}_l$ the subset of those having
length $l$.  We obtain a partition $\NN=\ZZ_{\geq0}=\cA^{(0)}\cup \cA^{(1)} \cup
\cA^{(2)} \cup\dots$.  We write $k_{b,d}(n)$ (or often simply $k(n)$) for the
number of occurrences of the digit $d$ in the minimal $b$-representation of the
integer $n$.  With this notation, the Irwin sums considered in this paper
are:
\[
H_{b,d}^{(k)} = \sum_{n>0,\; k_{b,d}(n)=k} \frac1n\;.
\]
As we are handling only positive terms, the order of summation is irrelevant.
The cardinality of the set of strings of $l$ $b$-ary digits, $k$ among them equal to
$d$, is $\binom{l}{k}(b-1)^{l-k}$.
Terms in $H_{b,d}^{(k)}$ with denominators of the same
length $l$ constitute ``blocks'' $S_{b,d;l}^{(k)}$:
\[
S_{b,d;l}^{(k)} = \sum_{l(n)=l,\; k_{b,d}(n)=k} \frac1n\;.
\]
If
$k>l$, no integer with $l$ digits can have $k$ occurrences of $d$, and
$S_{b,d;l}^{(k)} = 0$.  Note also that $H_{2,1}^{(0)} = 0$ because
the series is empty.
Taking into account that each $n^{-1}$ with $b^{l-1}\leq n<b^l$ contributes at
most $b^{-(l-1)}$, we see that $S_{b,d;l}^{(k)}=O_{l\to\infty}(l^k(1
-1/b)^l)$.  So, $H_{b,d}^{(k)}<\infty$.

As the base $b$ and the digit $d$ are fixed throughout the paper, we shall
usually drop them from the subscripts and abbreviate $H_{b,d}^{(k)}$ and
$S_{b,d;l}^{(k)}$ to, respectively, $H^{(k)}$ and $S_{l}^{(k)}$.
The notation $S_l^{(\leq k)}$ means $\sum_{0\leq j\leq k}S_l^{(j)}$, and
$S_l^{(< k)}=\sum_{0\leq j< k}S_l^{(j)}$.

Regarding integration against measures $\mu$, we will usually write $d\mu(x)$,
and sometimes $\mu(dx)$ with the same meaning.  The letter $d$ here is not a
$b$-ary digit...

\section{The New Series Representing Irwin Sums}

Here, we gather together the main results.
\begin{theorem}\label{thm:main}
  Let $b\in\NN$, $b>1$, and $d\in\{0,\dots,b-1\}$.
  Define, for each $j\geq1$,
  $\gamma_j= \sum_{\substack{a\neq d\\ 0\leq a <b}} a^j$.
  Let $u_{0;0}=b$ and let $(u_{0;m})_{m\geq0}$ satisfy the recurrence
  \begin{equation*}
   m\geq1\implies(b^{m+1} - b +1) u_{0;m} = \sum_{j=1}^{m} \binom{m}{j} \gamma_j u_{0;m-j}.
  \end{equation*}
  Let $u_{j;m}$ ($j\geq1$, $m\geq0$) satisfy the recurrence
  \begin{equation*}
    m\geq0\implies(b^{m+1} - b +1) u_{j;m}=\sum_{j=1}^{m}\binom{m}{j}\gamma_j u_{j;m-j}
    + \sum_{j=0}^{m}\binom{m}{j}d^j u_{j-1;m-j}.
  \end{equation*}
  In particular $u_{j;0} = u_{j-1;0} = \dots = u_{0;0} = b$.

  Let $k\geq 0$ and $l\geq1$.
  The Irwin sum $H^{(k)}$ can be expressed using integers of length at
  most $l$ and having at most $k$ occurrences of the digit $d$ (as indicated below by
  the superscript ${}^{(\leq k)}$):
  \begin{equation}\label{eq:main1}
    H^{(k)} = \sideset{}{^{(k)}}\sum_{0<n<b^{l-1}} \frac1n +
       b\cdot \sideset{}{^{(\leq k)}}\sum_{b^{l-1}\leq n<b^l} \frac1n  +
      \sum_{m=1}^\infty (-1)^{m}
       \;\sideset{}{^{(\leq k)}}\sum_{b^{l-1}\leq n<b^l}\; \frac{u_{k-k(n);m}}{n^{m+1}}\;.
  \end{equation}

  The quantities $u_{j;m}$ ($j\geq0$, $m\geq1$) have the following properties:
  \begin{itemize}
  \item They are non-negative and vanish only if $j=0$ and $b=2$ and $d=1$.
  \item They decrease (strictly if not zero) for increasing $m$.
  \item They increase strictly for increasing $j$ and converge to $b/(m+1)$.
  \end{itemize}
\end{theorem}
\begin{proof}
  This uses most everything from the present paper.  In brief: the validity of the
  series is established in Proposition \ref{prop:series}, which is a corollary
  to the integral formulas from Proposition \ref{prop:integralformula},
  themselves being variants of the $H^{(k)} =
  \int_{[b^{-1},1)}\frac{d\mu_k(x)}x$ {log-like} expression of
  Proposition \ref{prop:loglike}.  The coefficients $u_{j;m}$ are the moments
  of measures $\mu_j$ which are the main topic of this paper. The recurrence relations
  are established in Proposition \ref{prop:recurrence}.  The bounds on the
  $u_{j;m}$,  and their limits as $j\to\infty$, are obtained in
  Proposition \ref{prop:mainestimate}.  The
  $b/(m+1)$ limit for $j\to\infty$ is also a corollary of Proposition
  \ref{prop:convergence} which establishes the convergence of the measures
  $\mu_k$ to $b\,dx$.
\end{proof}
\begin{remark}
  We examine when the alternating series in Equation \eqref{eq:main1} has
  only vanishing contributions.  The involved coefficients $u_{j;m}$ must
  vanish, which can happen only for $(b,d,j)=(2,1,0)$.  Assume thus $b=2$ and
  $d=1$.  The contributing $n$ with $l$ digits are constrained by $k(n)\leq
  k$.  If $k=0$, there is no such $n$, and the alternating series is empty.
  If $k=1$, only $k(n)=1$ is realized (as a positive integer must have its
  leading digit equal to $1$), and only by $n=2^{l-1}$.  But then $j=k-k(n)=0$
  and thus $u_{j;m}=0$ and the contribution vanishes.  If $k>1$, the
  contribution of $n=2^{l-1}$ is $(-1)^mu_{k-1;m}/2^{(l-1)(m+1)}$ and is
  non-zero.  The alternating series is thus either identically vanishing
  ($b=2$, $d=1$, $k\leq1$) or its terms decrease strictly in absolute value.

  For $l=1$ the series is built with inverse powers of the digits.  If
  $k\geq1$ all digits contribute.  The first contributions to Equation \eqref{eq:main1}
  (left of the alternating series) add  up to $b(\frac11+\frac12+\dots +
  \frac1{b-1})$, which is thus an upper bound.  Also with $l=1$, if $k=0$ we
  have the same upper bound if $d=0$, and if $d>0$ we omit $b/d$
  as the integer $n=d$ gives no contribution.
  This upper bound $b\sum_{1\leq n<b} \frac1n -
  b\mathbf{1}_{k=0,d>0}(k,d)\frac1d$ is strict, except with $b=2$, $d=1$ and
  $k\leq 1$ all true in which case the upper bound is the exact value of
  $H^{(k)}$.
  In all other cases, keeping from the alternating series in Equation \eqref{eq:main1}
  only the first term, we obtain a strict lower bound of $H^{(k)}$.  See also
  Propositions \ref{prop:farhi} and \ref{prop:lowerbound}.
\end{remark}
\begin{remark}
  One can use the theorem with $l=1$, but if $k\geq1$, the one-digit number
  $n=1$ will then always contribute, and as all coefficients $u_{j;m}$ are
  bounded below if $d\neq b-1$ by $1/(m+1)$ (see Proposition
  \ref{prop:mainestimate}), we obtain for level $1$, $k\geq1$, $d\neq
  b-1$, a series as poorly converging as the one for $\log2 = 1 - \frac12 +
  \frac13-\frac14+\dots$.

  It is preferable for numerical implementations to use at least $l=2$.
  The
  finite sum of the reciprocal powers $1/n^{m+2}$ is bounded above by
  $b^l/(b^{l-1})^{m+2}$.  Combined with the $0\leq u_{l;m}\leq b/(m+1)$ bounds this gives
  an upper bound $b^{-m(l-1)+3-l}/(m+2)$.  Hence, each additional term of the series will
  give about $l-1$ new places of precision, in radix $b$ representation,
  for the approximation of the Irwin sum.
  Using $l=3$ has the advantage of dividing by two the needed
  range
  of $m$ for the same target precision.  It does induce additional cost
  in computing the inverse power sums, as they have more contributions.
  In our initial
  \href{https://www.sagemath.org}{\textsc{SageMath}} implementation from 2024%
  \footnote{Available at
    \url{https://arxiv.org/src/2402.09083v1/anc}. A more
    sophisticated version, applying parallelism to
    some extent, is now available at \url{https://gitlab.com/burnolmath/irwin}.
  }
  we
  observed for $b=10$ and $d=9$ that $l=3$ was beneficial at about $1200+$
  decimal digits for $H^{(0)}$ and already at $600+$ digits for $H^{(1)}$ and
  $400+$ digits for $H^{(2)}$ compared to using $l=2$.  But this depends on
  the actual implementation and on the numerical libraries used.
%
%
For small bases, the benefit of choosing $l=3$ --- and even $l=4$ in the cases $b=2$
and $b=3$ --- becomes evident already at substantially lower target precision.
\end{remark}

Theorem \ref{thm:mainpositive} below writes $H^{(k)}$ as a series of non-negative
terms that obey a linear recurrence relation and decay geometrically to
zero. The terms are positive except for $(b,d,k)=(2,1,0)$.
\begin{theorem}\label{thm:mainpositive}
  Let $b>0$, $d\in\{0,\dots, b-1\}$.
  For each $j\geq1$, let $\gamma_j'$ be $\sum_{\substack{a\neq b-1-d\\ 0\leq a<b}} a^j$.
  Let $(v_{0;m})_{m\geq0}$ satisfy the recurrence
  \begin{equation*}
    m\geq0\implies(b^{m+1} - b +1) v_{0;m} =
    b^{m+1} + \sum_{j=1}^{m} \binom{m}{j} \gamma_j' v_{0;m-j}.
  \end{equation*}
  In particular, $v_{0;0}=b$. Let $v_{j;m}$ ($j\geq1$, $m\geq0$) satisfy the recurrence
  \begin{equation*}
    m\geq0\implies(b^{m+1} - b +1) v_{j;m}=\sum_{j=1}^{m}\binom{m}{j}\gamma_j' v_{j;m-j}
    + \sum_{j=0}^{m}\binom{m}{j}(b-1-d)^j v_{j-1;m-j}.
  \end{equation*}
  In particular, $v_{j;0}=v_{j-1;0}=\dots = v_{0;0}=b$ for all $j\geq 0$.

  Let $l\geq1$ and $k\geq0$.  One has
  \begin{equation*}
    H^{(k)}= \sideset{}{^{(k)}}\sum_{0<n<b^{l-1}} \frac1n +
    b\cdot \sideset{}{^{(\leq k)}}\sum_{b^{l-1}\leq n<b^l} \frac1{n+1}  +
    \sum_{m=1}^\infty\;\quad
    \sideset{}{^{(\leq k)}}\sum_{b^{l-1}\leq n<b^l}\;\frac{v_{k-k(n);m}}{(n+1)^{m+1}}\;.
  \end{equation*}
  The superscript ${}^{(\leq k)}$ means to restrict to integers $n$ having in base $b$
  at most $k$ occurrences of the digit $d$.

  The quantities $v_{j;m}$ ($j\geq0$, $m\geq1$) have the following properties:
  \begin{itemize}
  \item They are positive and bounded above by $b$.
  \item They decrease strictly for increasing $m$, except for $b=2$, $d=1$, and
    $j=0$, in which case $v_{0;m}=2$ for all $m$.
  \item They decrease strictly for increasing $j$ and converge to $b/(m+1)$.
  \end{itemize}
\end{theorem}
\begin{proof}
  The series are established in Proposition \ref{prop:series}.  The
  recurrences are Equations \eqref{eq:recurrv} and \eqref{eq:recurrv0}.  The
  value $b/(m+1)$ of the limit of $(v_{j;m})_{j\geq0}$, and the monotonicity,
  are proven in Proposition \ref{prop:decreasev}.  Another proof
  follows from Proposition \ref{prop:convergence} about the convergence
  $\mu_j\to b\,dx$ and the definition of the $v_{j;m}$ as complementary
  power moments.
\end{proof}

The next statement was proven by Farhi \cite{farhi} for $b=10$.  We obtain it
here for all $b>1$.
\begin{proposition}\label{prop:farhi}
  With $d\neq0$ the sequence $(H^{(k)})$ is strictly decreasing for $k\geq1$.
  With $d=0$ it is strictly decreasing already starting at $k=0$.  In both
  cases the sequence converges to $b\log b$.
\end{proposition}
\begin{proof}
  We apply Theorem \ref{thm:mainpositive} with $l=1$.  Starting with $k=1$, there
  are no restrictions on the single-digit integers $n$ intervening in the
  inverse power sums. The theorem says in particular that the coefficients
  $v_{j;m}$ decrease when $j$ increases. This gives the decrease
  $(H^{(k)})_{k\geq1}$.  The theorem also gives their limits as $j\to\infty$. We
  recognize after taking these limits term by term the Taylor series of $-\log(1-h)$
  evaluated at $h=(n+1)^{-1}$, for $1\leq n <b$.  Adding up these logarithms
  gives a telescopic finite sum:
\[
  \lim H^{(k)} = b \sum_{1\leq n < b} -\log(1 - \frac1{n+1})  = b\log(b).
\]
  For $d=0$ all non-zero digits contribute to the series
  already for $k=0$. So the decrease starts already at $k=0$.
\end{proof}
\begin{proposition}\label{prop:lowerbound}
  For $b>2$ and $d\neq0$ one has $H^{(0)}> b\log(b) - b\log(1 +
  \frac{1}{d})$.  Hence $H^{(0)} > b\log(b/2)$.
\end{proposition}
\begin{proof}
  Same proof as for the previous proposition using $l=1$.  The sole difference
  is that only integers $n$ in $\{1,\dots,b-1\}\setminus\{d\}$ contribute to
  the series.
\end{proof}
\begin{remark}
  Hence, except for the sole case $b=2$, $d=1$, $k=0$, we have
  $H^{(k)}>\min(3\log(3/2),2\log(2)) = 3\log(3/2)>1.2$. This means that in a
  floating point context, we can decide of how many terms to keep in the series,
  solely on the basis of fixed point estimates, i.e., absolute comparison to $1$.
\end{remark}

\section{Irwin Sums as Integrals}

Recall that, given an enumerable subset $\{x_1,x_2,\dots\}$
of the real line, and a series with non-negative terms $\sum_{n\geq1}
c_n$, possibly diverging, one can define the set-function
$\mu:\mathcal{P}(\RR)\to \RR_{\geq0}\cup\{+\infty\}$ which assigns to
any subset $G$ of the real line the quantity $\mu(G) =
\sum_{n=1}^\infty c_n\mathbf{1}_{G}(x_n)\in[0,\infty]$.  This set-function is
$\sigma$-additive and is a (non-negative, discrete, possibly
infinite) measure.  Integrability of a complex-valued function $f$ on the real
line means in this context $\sum_{n=1}^\infty c_n |f(x_n)|<\infty$.
Then $\int_{\RR} f(x)\,d\mu(x)
$ is defined as $\sum_{n=1}^\infty c_n f(x_n)$,
and is invariant under any permutation of the indexing of the set
$\{(x_n,c_n), n\geq1\}$.  We will make free use of the notation of measures and integrals
in the following.  We write $\mu = \sum_{i\in\mathcal{I}} c_i\delta_{x_i}$, where the
countable index set $\mathcal{I}$ does not have to be $\NN$.
In this paper we consider only measures supported in $[0,1)$ and having finite
total mass.  The
\emph{support} is defined to be the set $\{x_i, c_i>0\}$ (not its closure in
the usual topology).  A \emph{Dirac mass} is a
measure $c\delta_x$ having  a single real number $x$ in its support. The
\emph{weight} is $c=\mu(\{x\})>0$.

\begin{definition}
  The measure $\mu_{b,d}^{(k)}$, or for short in the sequel $\mu_k$, is
  the (infinite) sum of Dirac masses at the rational numbers $x =
  n(X)/b^{|X|}$, for $X\in \cX^{(k)}$, with respective weights $1/b^{|X|}$:
\begin{equation*}
    \mu_{k} = \sum_{X\in \cX^{(k)}} b^{-|X|}\delta_{n(X)/{b^{|X|}}}\;.
\end{equation*}
It is supported in $[0,1)$.
\end{definition}
\begin{remark}
  In our previous work \cite{burnolkempner} of which the present paper is a
  continuation in the single-digit case, we
  defined a measure on $\RR_{\geq0}$ but the proofs of the main Theorem used
  only its restrictions to $[0,1)$.  We thus here define our
  measures to only have support on $[0,1)$ and leave aside considerations
  relative to what the ``correct'' extension to $\RR_{\geq0}$ is.
\end{remark}
\begin{remark}\label{rem:zero}
  Suppose $d\neq0$.  Then the strings with $n(X)=0$, i.e., the none-string
  and those containing only $0$'s, belong to $\cX^{(k)}$ if and only if $k=0$.
  For $k=0$ there is
  a total weight of $1 + b^{-1} + b^{-2} + \dots = b/(b-1)$ assigned to
  the Dirac at the origin.  But for $k\geq1$ and $d\neq0$, there is no such Dirac mass.

  If $d=0$, $n(X) = 0$ for $X\in \cX^{(k)}$ happens if
  and only if $X$ is the string consisting of $k$ zeros.
  So, in this case, there is always a Dirac mass at the
  origin, which has weight $b^{-k}$.
\end{remark}
In the next calculation displaying the binomial series, one has  $N=b-1$.
\begin{align*}
  \mu_{k}([0,1)) &= \sum_{l\geq0} b^{-l}\#\{ X \in \cX^{(k)} \cap \sD^l\} =
  \sum_{l\geq k} b^{-l}\binom{l}{k}N^{l-k} \\
 &= \sum_{p=0}^\infty b^{-k-p}\binom{k+p}{k} N^p
= b^{-k}\frac1{(1 - N/b)^{k+1}} = \frac{b}{(b-N)^{k+1}} = b.
\end{align*}
Hence, the measure $\mu_k$ is
finite for any $k\geq0$ and its total mass is $b$.

We can express Irwin numbers as log-like quantities.
In the next Proposition, recall that $\mu_k$ is short for $\mu^{(k)}_{b,d}$, i.e., it also depends on $(b,d)$.
\begin{proposition}\label{prop:loglike}
For $b\geq2$, $d$ a $b$-ary digit, and $k\geq0$, there holds 
    \begin{equation*}
      \label{eq:loglike}
      H^{(k)}_{b,d} = \int_{[b^{-1},1)}\frac{d\mu_k(x)}x\;.
    \end{equation*}
\end{proposition}
\begin{proof}
  The integral is the sum of $b^{-|X|}/(n(X)/b^{|X|}) = n(X)^{-1}$ over all
  strings $X$ containing exactly $k$ times the digit $d$, and having a non-zero
  leading digit.  Such strings are in one-to-one correspondance with positive
  integers, so this is $H^{(k)}$.
\end{proof}

We need some additional notation for the next proposition: for $n>0$ of length
$q\geq l$, we let $\ld_l(n)$ be the integer $m$ of length $l$ which is ``at start'' of
$n$, i.e., $m$ is the floor of $n/b^{q-l}$.
\begin{proposition}\label{prop:integralformula}
Let $n>0$ be an integer of length $l$. Let $k\geq0$. One has
\[
  \int_{[0,1)} \frac{d\mu_k(x)}{n+x} =
  \sideset{}{^{(k(n)+k)}}\sum_{\ld_l(m)=n}\frac1m\;.
\]
\end{proposition}
\begin{proof}
  The measure $\mu_k$ is defined as an infinite weighted sum of Dirac masses
  indexed by strings $X$ having exactly $k$ occurrences of the digit $d$.  We
  obtain
\[ \int_{[0,1)} \frac{d\mu_k(x)}{n+x} =
   \sum_{X\in \cX^{(k)}} \frac1{b^{|X|}(n + n(X)/b^{|X|})} =
   \sum_{X\in \cX^{(k)}} \frac1{ n\cdot b^{|X|} + n(X)}\;.
\]
The set of denominators present in this last sum is exactly the set of positive
integers with $\ld_l(m) = n$ and $k(m) = k(n) + k$.
\end{proof}
\noindent
If $d\neq0$, the contributions of any $x$ in the support of $\mu_k$ are of
the type $1/m$, $1/(bm)$, $1/(b^2m)$, \dots, as $0$ can be appended as
trailing element of a string $X$ without modifying the number of occurrences of $d$
nor the $b$-imal number $x=n(X)/b^{|X|}$.  If we index rather by the strings
as done here, there is a one-to-one correspondance, which proves more convenient.
The explanations in our earlier work \cite{burnolkempner} would have been a
bit simplified by this language, which however requires the additional notation
defined here.

As the positive measure with support in the unit interval is finite,
it has moments of all orders and these moments are the key quantities in our
analysis. We define, for $k\geq0$ and $m\geq0$:
\begin{equation*}
   u_{k;m} = \int_{[0,1)} x^m \,d\mu_k(x).
\end{equation*}
Except for the sole case of $k=0$, $b=2$, $d=1$, for which $\mu_0$ is a Dirac
mass
of weight $2$ at the origin, and $u_{0;m}=0$ for $m\geq1$, $(u_{k;m})_{m\geq0}$
is a strictly decreasing sequence converging to zero.
This follows by dominated convergence from $\mu_k$ being supported in $[0,1)$.  One can also argue elementarily as in [Proof of Theorem 4]{burnolkempner}.
\begin{corollary}\label{cor:alt}
  Let $k\geq0$ and let $n$ be a positive integer having $k(n)\leq k$
  occurrences of the digit $d$.  The contribution to $H_{b,d}^{(k)}$ from
  the denominators ``starting with (identical digits as in) $n$'' can be computed as an
  alternating series:
  \[ \sideset{}{^{(k)}}\sum_{\ld_l(m)=n}\frac1m = \sum_{m=0}^\infty
    (-1)^m \frac{u_{k-k(n);m}}{n^{m+1}}\;.\]
\end{corollary}
\begin{proof}
  We use the formula of Proposition \ref{prop:integralformula} with $k-k(n)$.
  For $n>1$, $1/(n+x)=1/n - x/n ^2 +x^2/n^3-\dots$ converges absolutely and
  uniformly with respect to $x\in [0,1)$ (even inclusive of $x=1$) and we can
  thus integrate term per term.  For $n=1$, the remainders after
  integration will be up to sign the integrals of $x^{m+1}/(1+x)$ and thus
  their absolute values are bounded by $\mu_{k-k(n);m+1}$, and they converge
  to zero.  Interchanging summation and integration is thus valid
  in that case too.
\end{proof}
\begin{remark}
  For $k=0$, $b=2$, $d=1$, there is no positive integer with $0$ occurrence of
  the digit $1$, so the statement is empty.
\end{remark}
We define \emph{complementary moments}
$v_{j;m}=\int_{[0,1)}(1-x)^m\,d\mu_j(x)$ for $j$, $m$ non-negative integers.
\begin{corollary}\label{cor:pos}
  Let $k\geq0$ and let $n$ be a positive integer
  of length $l$ and having at most $k$ occurrences of the digit $d$,
  i.e., $k(n)\leq k$.  Then
  \[\sideset{}{^{(k)}}\sum_{\ld_l(m)=n}\frac1m = \sum_{m=0}^\infty
    \frac{v_{k-k(n);m}}{(n+1)^{m+1}}\;.\]
  The right-hand side is a positive series with geometric convergence.
\end{corollary}
\begin{proof}
  This follows from
\[
  \frac1{n+x} = \frac1{n+1-(1-x)}= \sum_{m=0}^\infty
  \frac{(1-x)^m}{(n+1)^{m+1}}\;,
\]
and term-by-term integration on $[0,1)$ against the measure $\mu_k$.
\end{proof}
\begin{proposition}\label{prop:series}
  Let $k\geq0$ and $l\geq1$.  Then
  \begin{equation}\label{eq:irwink}
      H^{(k)} = \sum_{1\leq j < l} S_j^{(k)} + b S_l^{(\leq k)}
      + \sum_{m=1}^\infty (-1)^m \sum_{0\leq i\leq k}
      u_{k-i;m}\sum_{\substack{l(n)=l\\k(n)=i}}\frac{1}{n^{m+1}}\;.
  \end{equation}
  Equivalently
  \begin{equation*}
      H^{(k)}= \sum_{1\leq j < l} S_j^{(k)} + b S_l^{(\leq k)} +
      \sum_{m=1}^\infty (-1)^{m} \sum_{\substack{l(n)=l\\k(n)\leq k}} \frac{u_{k-k(n);m}}{n^{m+1}}\;.
  \end{equation*}
  One has similarly
  \begin{equation*}
      H^{(k)}= \sum_{1\leq j < l} S_j^{(k)} + b  \sum_{\substack{l(n)=l\\k(n)\leq k}} \frac{1}{n+1}+
      \sum_{m=1}^\infty \sum_{\substack{l(n)=l\\k(n)\leq k}} \frac{v_{k-k(n);m}}{(n+1)^{m+1}}\;.
  \end{equation*}
\end{proposition}
\begin{proof}
  Just apply Corollary \ref{cor:alt}, or Corollary \ref{cor:pos},
  to each integer $n$ of length
  $l$ and such that $k(n)\leq k$, then use $u_{k-i;0}=b=v_{k-i;0}$ to
  handle the contributions from $m=0$.
\end{proof}

\section{Integral Identities, Recurrence and Asymptotics of Moments}

The next lemma will allow us to obtain recurrence formulas for the moments.
\begin{lemma}\label{lem:int}
  Let $f$  be a bounded function on $[0,b)$.  Let $k\geq1$.
  One has
  \begin{equation*}
    \int_{[0,1)} f(bx)\,d\mu_k(x) = \frac1b\sum_{a\neq d}\int_{[0,1)} f(a+x)\,d\mu_k(x) +
                                  \frac1b \int_{[0,1)} f(d+x)\,d\mu_{k-1}(x),
  \end{equation*}
  where the first summation is over all digits distinct from $d$.  For $k=0$,
  one has
\begin{equation}\label{eq:intmu0}
\int_{[0,1)} f(bx)\,d\mu_0(x) = f(0) + \frac1b\sum_{a\neq d}\int_{[0,1)} f(a+x)\,d\mu_0(x).
\end{equation}
\end{lemma}
\begin{remark}
  Equation \eqref{eq:intmu0} was already stated in \cite[Lemma 7]{burnolkempner}
  except for
  $b=2$ and $d=1$: indeed, reference \cite{burnolkempner} has a set $\sA\subset\sD$ of
  so-called admissible digits and assumes that this set is not reduced to the
  singleton $\{0\}$, which is however what happens for $(b,d,k)=(2,1,0)$.  As the measure
  $\mu_0$ then equals twice the Dirac at the origin,
  Equation \eqref{eq:intmu0} reads in that special case
  $2f(0)=f(0)+\frac12 (2f(0))$.
\end{remark}
\begin{proof}[Proof of Lemma \ref{lem:int}]
  Suppose $k\geq1$. Each string $X$ of length $l$ and having exactly $k$
  occurrences of $d$ contributes $b^{-l}f(n(X)/b^{l-1})$.  For $l=0$, $X$ is
  the none-string and does not contribute anything as we have supposed
  $k\geq1$.  For $l\geq1$ let $a$ be the leading digit of $X$.  If $a\neq d$,
  then the $l-1$ remaining digits of $X$ give a string $Y$ which again has
  exactly $k$ occurrences of $d$.  So the strings with initial digit
  $a\neq d$ contribute the sum of the $b^{-1}b^{1-l}f(a+n(Y)/b^{l-1})$ over
  all $Y$ with $k$ occurrences of $d$.  This is $b^{-1}\int_{[0,1)}
  f(a+x)\,d\mu_k(x)$.  If $a=d$, then $Y$ (which may be the none-string) has
  $k-1$ occurrences of $d$.  Hence, this contributes $b^{-1}\int_{[0,1)}
  f(d+x)\,d\mu_{k-1}(x)$.

  We also consider $k=0$.  Here the none-string contributes $f(0)$ to the
  integral on the left-hand side.  The strings $X$ of length $l \geq1$ in
  $\cX^{(0)}$ do not contain the digit $d$.  So here the possible $a$ are
  distinct from $d$, and the tail string $Y$, possibly the
  none-string, automatically also belongs to $\cX^{(0)}$.  Hence, we have
  Equation \eqref{eq:intmu0}.
\end{proof}

Recall from \cite[Proposition 8]{burnolkempner} the recurrence
(obtained as a corollary to Equation \eqref{eq:intmu0}):
  \begin{equation}\label{eq:recurr0}
    m\geq1\implies (b^{m+1} - b +1) u_{0;m} =
   \sum_{j=1}^{m} \binom{m}{j} (\sum_{a\neq d} a^j) u_{0;m-j}.
  \end{equation}
We note that this formula requires $m\geq1$.  It also works in
the $k=0$, $b=2$, $d=1$ case, as all moments for $m\geq1$
then vanish and the power sum for $j=m\geq1$ is $0^m=0$.  We obtain
recurrences for the moments of the measures $\mu_k$, $k\geq1$.
\begin{proposition}\label{prop:recurrence}
  Let $k\geq1$.  The moments $u_{k;m}$ of the $k$-th measure $\mu_k$ are
  related to those of $\mu_{k-1}$ via the recurrence
  \begin{equation}\label{eq:recurr}
    (b^{m+1} - b +1) u_{k;m}=\sum_{j=1}^{m}\binom{m}{j}(\sum_{a\neq d} a^j)u_{k;m-j}
    + \sum_{j=0}^{m}\binom{m}{j}d^j u_{k-1;m-j}.
  \end{equation}
  One has $u_{k;0} = b$ and the above identity also holds for $m=0$.
\end{proposition}
\begin{proof}
  We apply Lemma \ref{lem:int} to the function $f(x) = x^m$, $m\geq1$.  After
  multipliying by $b$ we have on the left-hand side $b^{m+1}u_{k;m}$.  On the
  right-hand side, we apply the binomial formula (which requires the
  convention $0^0=1$) and separate the $j=0$ contribution from the first
  sum, obtaining
  \begin{align*}
    &\sum_{a\neq d} \sum_{j=0}^{m}\binom{m}{j}a^j u_{k;m-j}
   + \sum_{j=0}^{m}\binom{m}{j}d^j u_{k-1;m-j}\\
     &=(b-1)u_{k;m}+\sum_{j=1}^{m}\binom{m}{j}(\sum_{a\neq d} a^j)u_{k;m-j} + \sum_{j=0}^{m}\binom{m}{j}d^j u_{k-1;m-j}.
  \end{align*}
  This gives the stated formula.
\end{proof}

The following result is important both for theory and practice.
\begin{proposition}\label{prop:mainestimate}
  For each $m\geq1$ the sequence $(u_{k;m})_{k\geq0}$ is strictly increasing and
  converges to $b/(m+1)$:
  \begin{equation*}
      u_{0;m}< u_{1;m}<\dots < u_{k;m} \mathop{\longrightarrow}\limits_{k\to\infty} \frac{b}{m+1}\;.
  \end{equation*}

  Let $f=\max(\sD\setminus\{0,d\})$.  The case $f=0$ happens only if $b=2$ and
  $d=1$.  Suppose $f > 0$, then for all $m\geq1$ one has
  \begin{equation}\label{eq:kzerobound}
      \frac{1}{m+1}(\frac{f}{b-1})^{m}<u_{0;m}<\frac{b}{m+1}(\frac{f}{b-1})^{m}.
  \end{equation}
  If $b=2$ and $d=1$ one has $u_{0;m}=0$ for $m\geq1$ and
  $u_{1;m}=2/(2^{m+1}-1)$ for $m\geq0$.
\end{proposition}
\begin{proof}
  If $b=2$ and $d=1$ the measure $\mu_0$ is $2\delta_0$.  The values of
  $u_{1;m}$ are directly given by the recurrence from Equation \eqref{eq:recurr}
  whose right-hand side in that case only has a single non-zero contribution,
  which is $u_{0;0}=2$ hence the value for $u_{1;m}$.

  The estimate from Equation \eqref{eq:kzerobound} of $u_{0;m}$ (for either $b>2$ or
  $d\neq1$) is from \cite[Proposition 10]{burnolkempner}.

  We prove $u_{1;m}>u_{0;m}$ for all $m\geq1$.  This is already known for
  $b=2$ with $d=1$ as $u_{0;m} = 0$ so we exclude this case in the next
  paragraph.

  We compare the recurrence of the $(u_{1;m})$ sequence (Equation \eqref{eq:recurr}) with
  the one of the $(u_{0;m})$ sequence (Equation \eqref{eq:recurr0}).  They look the same
  apart from the fact that $k=1$ has more contributions, all non-negative.  So
  $u_{1;m}\geq u_{0;m}$ by induction on $m$ (as it is an equality for $m=0$).
  Reexamining Equation \eqref{eq:recurr} we see that the second sum on its
  right-hand side always contains the $j=0$ contribution $u_{0;m}$ which is
  positive.  So in fact in the previous argument we had $(b^{m+1}-b+1)u_{1;m}>
  (b^{m+1}-b+1) u_{0;m}$ for $m\geq1$, hence $u_{1;m}>u_{0;m}$.

  Let $k\geq1$ and suppose we have shown already $u_{k;m}>u_{k-1;m}$
  for all $m\geq1$.

  We consider Equation \eqref{eq:recurr} for $k+1$.  We can suppose
  inductively that $u_{k+1;n}\geq u_{k;n}$ for $0\leq n < m$ as this holds for
  $n=0$. And we know $u_{k;n}\geq u_{k-1;n}$ for all $n$.  Using this we
  obtain a lower bound $(b^{m+1}-b+1)u_{k;m}$ for the right-hand side of Equation
  \eqref{eq:recurr}
  with $k+1$.  Hence,
  $u_{k+1;m}\geq u_{k;m}$.  Thus, this holds for
  all $m$ by induction on $m$.  Reexamining Equation \eqref{eq:recurr} for $k+1$ we see
  that the last sum has the contribution for $j=0$ which is $u_{k;m}$ which
  is known to be greater than $u_{k-1;m}$. So in fact our lower bound is strict and
  $u_{k+1;m}> u_{k;m}$ for all $m\geq1$.  Hence, the conclusion by induction
  on $k$.

  So, for each $m\geq1$, there holds $u_{0;m}< u_{1;m}<\dots < u_{k;m} <
  \dots$.  Further, all are bounded above by $b$ as $u_{k;m}\leq u_{k;0}=b$. So,
  there exists a finite limit $w_{m}=\lim_{k\to\infty} u_{k;m}$.  Letting
  $k\to\infty$ in Equation \eqref{eq:recurr} we obtain for all $m\geq0$:
  \begin{equation*}
    (b^{m+1} - b +1) w_{m}=\sum_{j=1}^{m}\binom{m}{j}(\sum_{a=0}^{b-1} a^j)w_{m-j} + w_m.
  \end{equation*}
  As $u_{k;0}=b$ for all $k\geq0$, $w_0=\lim u_{k;0} = b$.  We prove by induction
  that $w_m =b/(m+1)$ holds for all $m\geq0$.  Assume it is true up to
  $m=M-1$ for some $M\geq1$.  Substituting this into the recurrence relation above (after having
  removed from both sides one copy of $w_M$) leads to
  \begin{equation*}
     (b^{M+1} - b) w_{M}=\sum_{j=1}^{M}\binom{M}{j}(\sum_{a=0}^{b-1} a^j)\frac{b}{M-j+1}\;.
  \end{equation*}
  Note that
  \begin{align*}
    \sum_{j=1}^{M}\binom{M}{j}(\sum_{a=0}^{b-1} a^j)\frac{M+1}{M+1-j} &=
   \sum_{j=1}^{M}\binom{M+1}{j}(\sum_{a=0}^{b-1} a^j)\\
   &= \sum_{a=0}^{b-1} \left((a+1)^{M+1} - a^{M+1} -1\right) = b^{M+1} - b.
  \end{align*}
  So $(M+1)(b^{M+1} - b) w_{M}$ is equal to $b(b^{M+1}
  - b)$.  Hence, $w_M=b/(M+1)$ and this completes the proof.
\end{proof}
\begin{remark}
Defining for all $k\geq0$ and all $m\geq0$
\[ \sigma_{k;m}=(m+1)u_{k;m},\]
 we obtain from Equation \eqref{eq:recurr} ($k\geq1$)
    \begin{equation*}
      (b^{m+1} - b +1) \sigma_{k;m}=
      \sum_{j=1}^{m}\binom{m+1}{j}\sum_{a\neq d} a^{j}\sigma_{k;m-j}
      +\sum_{j=0}^{m}\binom{m+1}{j}d^{j} \sigma_{k-1;m-j}.
  \end{equation*}
  We note that this is a barycentric equality with non-negative coefficients:
  \begin{align*}
      \sum_{j=1}^{m}\binom{m+1}{j}(\sum_{a\neq d} a^j)
      &+ \sum_{j=0}^{m}\binom{m+1}{j}d^j  \\
&= \sum_{a=0}^{b-1} \Bigl((a+1)^{m+1}-a^{m+1} - 1\Bigr) + 1=b^{m+1}-b + 1.
 \end{align*}
  An alternative proof of $\frac{1}{m+1}(\frac{f}{b-1})^{m} < u_{k;m} <
  \frac b{m+1}$ can be based upon this.
\end{remark}

\section{Convergence to Lebesgue Measure and Farhi Theorem}

In terms of $g(x) = f(bx)$, the integral formula of Lemma \ref{lem:int}  becomes
  \begin{equation*}
    \int_{[0,1)} g(x)\,d\mu_k(x) =
    \frac1b\sum_{a\neq d}\int_{[0,1)} g(\frac ab+\frac xb)\,d\mu_k(x) +
    \frac1b \int_{[0,1)} g(\frac db +\frac xb)\,d\mu_{k-1}(x).
  \end{equation*}
  This motivates a closer examination of the restrictions of $\mu_k$ to
  sub-intervals such as $[i/b,(i+1)/b)$ for $0\leq i < b$.
\begin{lemma}\label{lem:masses}
  Let $x\in[0,1)$ be a $b$-imal number of depth $l$.  Let $X$ be the string of
  length $l$ such that $n(X)/b^l = x$.
  Let $j$ be the number of occurrences of $d$ in $X$.  Let $l'\geq l$.  Set
  $j'=j$ if $d>0$, and $j'=j+l'-l$ if $d=0$.
  Let $U$ be any subset of the
  open interval $(x,x+b^{-l'})$.  Finally, let $k\in\NN$.
    \begin{itemize}
    \item if $k<j'$ then $\mu_k(U)=0$,
    \item if $k\geq j'$ then $\mu_k(U) = b^{-l'}\mu_{k-j'}(b^{l'} U - b^{l'}
      x)$.
    \end{itemize}
\end{lemma}
\begin{proof}
  If $d>0$,
  $j$ is also the number of occurrences of $d$ in the integer $b^l
  x$, but for $d=0$, $j$ will be greater than that if $x<b^{-1}$ (due to leading
  zeros in $X$).

  Any string $Y$ such that $n(Y)/b^{|Y|}\in (x, x+b^{-l'})$
  has the shape
  \[ Y = X\underbrace{0\dots 0}_{l'-l\text{ zeros}}Z\]
  where $Z=z_1\dots z_p$ and at least one $z_i$ is not
  zero.
  The number of occurrences of $d$ in $Y$ is the sum of $j'$ (which was
  defined depending on whether $d>0$ or $d=0$) with the number of occurrences
  in $Z$.  So if $k<j'$ no such string $Y$ has exactly $k$ occurrences of $d$
  and $\mu_k$ restricts to the zero measure on $(x,x+b^{-l'})$.
  Suppose $k\geq j'$.  Set $y=n(Y)/b^{|Y|}$ and $z=n(Z)/b^{|Z|}$. Thus, $y= x
  + z/b^{l'}$ with $0<z<1$, and $z=b^{l'}(y -x)$. Conversely, any string $Z$
  having at least one non-zero digit can be extended as above to give $Y$
  such that $n(Y)/b^{|Y|}$ is in $(x, x+b^{-l'})$ and $k(Y)=k(Z)+j'$ ($k(T)$
  is the number of occurrences of $d$ in a string $T$). Summing
  over all $Y$ with $n(Y)/b^{|Y|} = y$ and having $k$ occurrences of $d$ (if
  $d=0$, there is only one such $Y$ for each $y$), and over all $Z$ giving
  the same $z$ and having $k-j'$ occurrences of $d$, we get
  $\mu_k(\{y\})=b^{-l'}\mu_{k-j'}(\{z\})$.  Finally, summing over all strings
  $Y$ such that $n(Y)/b^{|Y|}\in U\subset (x, x+b^{-l'})$, we obtain the
  stated formula for $\mu_k(U)$.
\end{proof}
For some half-open intervals, a simple formula showing the behavior of $\mu_k$
is obtained next.
\begin{proposition}\label{prop:halfopen}
  Let $x\in[0,1)$ be of depth $l$ and let $j$ be defined as in the previous lemma.
  For $k<j$ the restriction of $\mu_k$ to the half-open interval
  $[x,x+b^{-l})$ vanishes.  For $k\geq j$ and any subset $U\subset
  [x,x+b^{-l})$ one has $\mu_k(U)=b^{-l}\mu_{k-j}([b^l U-b^l x))$.
\end{proposition}
\begin{proof}
  We use the Lemma \ref{lem:masses} with  $l'=l$, $j'=j$.
  For $k<j$ one thus has
  $\mu_k((x,x+b^{-l}))=0$.  And $\mu_k(\{x\})$ is also zero because there are
  already $j>k$ occurrences of the digit $d$ in $x$ (this includes leading
  zeros located after the radix separator).  So $\mu_k([x,x+b^{-l}))=0$.

  Suppose $k\geq j$. We know from Lemma \ref{lem:masses} that
  for any $U\subset (x,x+b^{-l})$,  $\mu_k(U)=b^{-l}\mu_{k-j}(b^lU-b^lx)$.
  There remains to examine what happens for the singleton $\{x\}$.  Let
  $c(k-j)=\mu_{k-j}(\{0\})$.  We need to check that $\mu_k(\{x\}) = b^{-l}
  c(k-j)$.  The value of $\mu_k(\{x\})$ depends
  on whether $d=0$ or $d>0$.  In the former case, $\mu_k(\{x\})=b^{-l-(k-j)}$
  (recall $k\geq j$).  In the latter case, it is equal to
  $b^{-l}(1 - 1/b)^{-1}$ if $k=j$, and vanishes if $k>j$.   Using Remark
  \ref{rem:zero}, we obtain $\mu_k{\{x\}}= b^{-l}c(k-j)$ in all cases, which
  completes the proof.
\end{proof}
\begin{remark}
  In particular the total mass $\mu_k([x,x+b^{-l}))$ is $0$ for $k<j$ and
  $b^{1-l}$ (i.e., $b$ times Lebesgue measure) for $j\geq k$.  The sequence
  $(\mu_k([x,x+b^{-l})))_{k\geq0}$ is thus non-decreasing.

  This is compatible with the moments being increasing as $k$ increases (cf.\@ Proposition
  \ref{prop:mainestimate}).   But the sequences $(\mu_k([t,u))_{k\geq0}$ associated
  with half-open intervals $[t,u)$ can not possibly all be non-decreasing: if
  they were, Equation
  \ref{eq:loglike} from Proposition \ref{eq:loglike} expressing $H^{(k)}$ as a
  log-like integral would cause $(H^{(k)})$ to be as well a non-decreasing
  sequence. But as first proven by
  Farhi \cite{farhi} for $b=10$, and generally here in Proposition
  \ref{prop:farhi}, they actually decrease strictly for $k\geq1$.

  This apparent paradox is explained by the fact that for $l'>l$,
  $\mu_k([x,x+b^{-l'}))$ has a less simple behavior when $k$ varies, than the
  one which is valid for $l'=l$ and described in Proposition
  \ref{prop:halfopen}. The details can be deduced from Lemma \ref{lem:masses}.
  We will only need that starting with $k= j'$, the sequence
  $(\mu_k([x,x+b^{-l'}))$ becomes constant, equal to $b^{1-l'}$.  This follows
  from the next lemma.
\end{remark}
\begin{lemma}\label{lem:l'}
  Let $x\in[0,1)$ be a $b$-imal number of depth $l$ for some $l\in\NN$.  Let
  $l'\geq l$ and let $j$ and $j'$ be defined as in Lemma \ref{lem:masses}.
  Let $k\geq
  j'$.  Let $U$ be any subset of the half-open interval $[x,x+b^{-l'})$.  Then
  $\mu_k(U) = b^{-l'}\mu_{k-j'}(b^{l'} U - b^{l'} x)$.
\end{lemma}
\begin{proof}
  We know this already from Lemma \ref{lem:masses} if $U$ does not contain
  $x$. It remains to consider the case of $U=\{x\}$.
  \begin{itemize}
  \item If $d=0$: for $k\geq j$, there is only one string $X$ containing $k$
    occurrences of the digit $0$ and such that $n(X)/b^{|X|} = x$. It is obtained
    by adding $k-j$ trailing zeros to the string of length $l$ representing
    $b^l x$.  So $\mu_k(\{x\}) = b^{-l-(k-j)}$. On the other hand $j'=j+l'-l$
    so $k-j'=k-j+l-l'$ and the weight of $\delta_0$ in $\mu_{k-j'}$ is
    $b^{-(k-j')}$ by Remark \ref{rem:zero}. So $b^{-l'}\mu_{k-j'}(\{0\}) =
    b^{-l'-k+j'} = b^{-k+j-l}$ which matches $\mu_k(\{x\})$.
  \item If $d\neq0$: for $k=j$, strings with trailing zeros contribute to the
    weight at $x$, and the total weight is $b^{-l}\sum_{i\geq0}b^{-i}$.  On the other hand
    $j'=j$ so $k-j'=k-j=0$ and the weight of $\delta_0$ in $\mu_0$ is
    $\sum_{i\geq0}b^{-i}$. Multiplying this by $b^{-l'}$ we obtain indeed
    $\mu_j(\{x\})$.  For $k>j$, $\mu_k(\{x\})=0$.  And also
    $b^{-l'}\mu_{k-j}(\{0\})=0$.  Again the values match.\qedhere
  \end{itemize}
\end{proof}

\begin{proposition}\label{prop:convergence}
  Let $t<u$ be any two $b$-imal numbers in $[0,1]$. Let $l'$ be large enough
  for $b^{l'}t$ and $b^{l'}u$ to be integers. For $k\geq l'$ there holds
\[\mu_k([t,u))= bu - bt.\]
  Let generally $I$ be any sub-interval of $[0,1)$.  Then
  (with $|I|$ defined as $\sup I - \inf I$, i.e., the Lebesgue measure of $I$)
\[
\lim \mu_k(I) = b|I|.
\]
\end{proposition}
\begin{proof}
  Let $t=n/b^{l'}$, $u=m/b^{l'}$ for some integers $0\leq n < m\leq b^{l'}$.
  It is enough to consider the case $m=n+1$, by additivity.
  Lemma \ref{lem:l'} says in particular that $\mu_k([t,t+b^{-l'})) =
  b^{-l'}\mu_{k-j'}([0,1))$ for $k\geq j'$.  The quantity $j'$ is
  here some
  integer at most equal to $l'$. And $\mu_i([0,1))=b$ for all $i\in\NN$.
  This gives
  the result for $\mu_k([t,u))$.  Note that here perhaps
  $l'>l$, where $l$ is the smallest integer such that $b^l t\in \NN$, and
  Proposition \ref{prop:halfopen} would not have been enough to conclude in that case.

  Let $I$ be any sub-interval of $[0,1)$.  If $I$ is a singleton the statement
  is known from evaluations of $\mu_k({x})$ (which is zero if $x$ is not
  $b$-imal).  If $I$ isn't a singleton, then
  $\liminf \mu_k(I)\geq \lim \mu_k([t,u)) = bu-bt$ for any choice of $b$-imal
  numbers $t<u$ in the interior $\mathring I$.  So $\liminf \mu_k(I)\geq b|I|$ and the upper
  bound for $\limsup \mu_k(I)$ is shown similarly.
\end{proof}

\begin{theorem}\label{thm:localfarhi}
  Let $n>0$ be an integer.  Then
\[
  \lim_{k\to\infty} \;\;\sideset{}{^{(k)}}\sum_{m\text{ \upshape starts with }n}\frac1m = b \log(1 +
  \frac1n).
\]
\end{theorem}
\begin{proof}
  The condition on $m$ is $\ld_l(m)=n$, where $l$ is the number of digits in
  $n$.  Using Proposition \ref{prop:integralformula}, we are thus looking at
\[
\lim_{k\to\infty} \int_{[0,1)} \frac{\mu_k(dx)}{n+x}\;.
\]
According to the previous proposition and familiar arguments from
measure theory this limit exists and its value is
\[
\int_{[0,1)} \frac{b\,dx}{n+x} = b\log(1 + \frac 1n).\qedhere
\]
\end{proof}
\begin{remark}
  Adding these formulas for $n=1$ to $b-1$ we recover Farhi theorem
  \cite{farhi}: $\lim H^{(k)} = b\log b$.  This can also be obtained from
  taking the limit in Equation \eqref{eq:loglike}. Expanding using the power
  series for the logarithm function we obtain a series which is the limit term
  per term of the one from Corollary \ref{cor:alt}.
\end{remark}

\section{Complementary Moments}

Let $E_k(t) = \int_{[0,1)} e^{tx}\,d\mu_k(x)$ be the exponential generating
function of the moments.  Using Lemma \ref{lem:int}, one obtains:
\begin{equation*}
  E_k(bt) = \begin{cases}\frac1b\left(\sum_{a\neq d} e^{at}E_k(t) + e^{dt} E_{k-1}(t)\right)&(k>0),\\
  1 + \frac1b\sum_{a\neq d} e^{at}E_0(t) & (k=0).\end{cases}
\end{equation*}
Define $F_k(t) = e^t E_k(-t)$.  This is the exponential generating functions of
the $v_{k;m}$ ($m\geq0$) which are defined right before Corollary \ref{cor:pos}.  Setting
$d'=b-1-d$, one has:
\begin{align*}
  F_k(bt) &= \frac1b\left(\sum_{a\neq d'} e^{at}F_k(t) + e^{d't} F_{k-1}(t)\right),\\
  F_0(bt) &= e^{bt} + \frac1b\sum_{a\neq d'} e^{at}F_0(t).
\end{align*}
Thus, $v_{k;m}$ obey for $k\geq1$ the same recurrences stated in
Proposition \ref{prop:recurrence} for $u_{k;m}$, except for the
replacement of $d$ by $d'=b-1-d$.  In other terms, setting
 $\gamma_j' = \sum_{a\neq d', 0\leq a <b} a^j$, we have:
  \begin{equation}\label{eq:recurrv}
    (b^{m+1} - b +1) v_{k;m}=\sum_{j=1}^{m}\binom{m}{j}\gamma_j'v_{k;m-j}
    + \sum_{j=0}^{m}\binom{m}{j}(d')^j v_{k-1;m-j}.
  \end{equation}
And
for $k=0$ we get:
\begin{equation}\label{eq:recurrv0}
(b^{m+1} - b +1)v_{0;m} = b^{m+1} + \sum_{j=1}^m \binom{m}{j}\gamma_j'v_{0;m-j}.
\end{equation}
The previous formula is also valid for
$m=0$: it gives $v_{k;0} = b$.
\begin{proposition}\label{prop:decreasev}
  For each $m\geq1$, the sequence $(v_{k;m})_{k\in\NN}$ (which is bounded
  above by $b=v_{k;0} = u_{k;0}$) is strictly decreasing and converges to $b/(m+1)$.
\end{proposition}
\begin{proof}
  First of all $v_{k;m}=\int_{[0,1)}(1-x)^m\,d\mu_k(x)\leq v_{k;0}=b$ with
  equality possible for $m\geq1$ only if all mass of $\mu_k$ is concentrated
  at the origin. This happens if and only if $b=2$ and $d=1$ and $k=0$.  Using
  the inequality we get
\[ \sum_{j=0}^{m}\binom{m}{j}(d')^j v_{k-1;m-j} \leq
   \sum_{j=0}^{m}\binom{m}{j}(b-1)^j b = b^{m+1}.
\]
For $d'=0$, the left-hand side is $v_{k-1;m}$ and it is less than $b^{m+1}$ for
$m\geq1$. For $d'>0$, we are not in the case $b=2$ and $d=1$, so
$v_{k-1;m-j}<b$ for $j<m$.  Consequently, and looking in particular at the
$j=0$ contribution, we have if $m\geq1$ a strict inequality
in the above equation.

Hence, for $k\geq1$, and $m\geq1$, from Equation \eqref{eq:recurrv}:
\[
 (b^{m+1} - b +1) v_{k;m}< \sum_{j=1}^{m}\binom{m}{j}\gamma_j'v_{k;m-j}
    + b^{m+1}.
\]
If we had equality here, this would be the same recurrence with the same
starting point as for $v_{0;m}$.  As we have an inequality, we get by
induction on $m$ $v_{k;m}\leq v_{0;m}$.  But then,
the right-hand side is $\leq (b^{m+1}-b+1)v_{0;m}$, so we get
$v_{k;m}<v_{0;m}$ for $m\geq1$.


Let $k=2$.
We obtain, from Equation \eqref{eq:recurrv}, the upper bound
\[
(b^{m+1} - b +1) v_{2;m}\leq \sum_{j=1}^{m}\binom{m}{j}\gamma_j'v_{2;m-j}
    + \sum_{j=0}^{m}\binom{m}{j}(d')^j v_{0;m-j}.
\]
If the less-than-or-equal sign was replaced with an equality, it would be the recurrence which applies to $(v_{1,m})$.  As the
two sequences have the same $m=0$ value we get $v_{2;m}\leq v_{1;m}$ for all
$m$.  But in the second sum in Equation \eqref{eq:recurrv} with $k=2$ we have the $j=0$
term which is $v_{1;m}$, and it is known to be less than $v_{0;m}$ if $m\geq1$.  So
the above displayed inequality is strict for $m\geq1$.  Using then in
the right-hand side $v_{2;m-j}\leq v_{1;m-j}$, and Equation \eqref{eq:recurrv} with
$k=1$, we get $(b^{m+1}-b+1)v_{2;m}<(b^{m+1}-b+1)v_{1;m}$ hence $v_{2;m} <
v_{1;m}$ for all $m\geq1$.

This argument can be repeated inductively and establishes that
$(v_{k;m})_{k\geq0}$ is strictly decreasing for each fixed $m\geq1$.

The value of the limit as $k\to\infty$ can be established as in the proof of
Proposition \ref{prop:mainestimate}.  Alternatively, the second paragraph of
Proposition \ref{prop:convergence} says that the sequence of probability
measures $(b^{-1}\mu_k)_{k\geq0}$ converges weakly to the Lebesgue measure on
the interval $[0,1)$. This implies the convergence of the complementary moments to those of the Lebesgue measure: $\lim_{k\to\infty} b^{-1}v_{k;m} = (m+1)^{-1}$.
\end{proof}

\begin{remark}
  For the notion of weak convergence of measures, especially of probability
  measures on the real line or some interval, see \cite[section 25]{billingsley}.
\end{remark}

\cleardoublepage
\normalsize
\section*{Corrigendum}
\begin{center}
\uppercase{\bf \boldmath Corrigendum to \#A11 of volume 26, Measures for the summation of Irwin series}
\vskip 20pt
{\bf Jean-Fran\c{c}ois Burnol}\\
{\it D\'epartement de Math\'ematiques, Universit\'e de Lille, France}\\
{\tt jean-francois.burnol@math.univ-lille.fr}
\end{center}
\vskip 20pt

\centerline{\small published Oct. 7, 2026}
\vskip 30pt

At the bottom of page 15, the statement of Lemma~3 is incorrect.  Replace its
claim
\[
  \mu_k(U)
  =
  b^{-l'}\mu_{k-j'}(b^{l'}U-b^{l'}x)
\]
with
\[
  \mu_k(U)
  =
  b^{-l'}\mu_{k-j'}(b^{l'}U-b^{l'}x)
  +
  \delta_{k,j}\,\mathbf{1}_U(x)\,
  \frac{b}{b-1}\left(b^{-l}-b^{-l'}\right).
\]
The extra term arises only if $x\in U$.  For $d=0$, the extra term vanishes:
if $l'>l$, then $j'>j$, and $k\geq j'$ forces $\delta_{k,j}=0$.  For $d>0$,
the error in the proof on page 16 is in the sentence ``Multiplying this
by $b^{-l'}$ we obtain indeed $\mu_j(\{x\})$''.  This assumes $l'=l$,
apparently from a confusion with $j'=j$.  For $l'>l$, when $U=\{x\}$ or,
more generally, $x\in U$, the argument in the proof leads to the above extra
term.

The sentence immediately preceding Lemma~3 should be modified: replacing
``starting with $k=j'$'' with
``starting with $k=\max(j',j+1)$'' makes it true.

The only subsequent use of Lemma~3 is in the proof of Proposition~9.
There, $k\geq l'$. Hence, if $l'>l$, then
$k\geq l'>l\geq j$, so the problematic case $k=j$, $l'>l$ does not occur.

\bigskip
The paper contains a number of typographical infelicities of varying severity:
\begin{itemize}

\item
On page~3, in Section~2, the map from strings to
integers is written as
\[
n(X)=d_l b^l+\dots+d_1
\qquad\text{for }X=(d_l,\dots,d_1).
\]
The leading term should of course be $d_l b^{l-1}$.

\item
On page~4, in Theorem~1, in the recurrence defining $u_{j;m}$ for $j\geq1$,
the letter $j$ is used both for the fixed first index of $u_{j;m}$ and as
the summation index in both sums. This is unfortunate; the right-hand side should read
\[
\sum_{i=1}^{m}\binom{m}{i}\gamma_i u_{j;m-i} + \sum_{i=0}^{m}\binom{m}{i}d^i u_{j-1;m-i}.
\]
The same problem occurs on page~6 in Theorem~2, in the recurrence defining
$v_{j;m}$, whose right-hand side should read
\[
\sum_{i=1}^{m}\binom{m}{i}\gamma'_i v_{j;m-i} + \sum_{i=0}^{m}\binom{m}{i}(b-1-d)^i v_{j-1;m-i}.
\]

\item
On page~5, in Remark~2, the text says ``Combined with the $0\leq
    u_{l;m}\leq b/(m+1)$ bounds''. Replace $l$ with $j$.

\item
On page~5, Theorem~2 starts with
``Let $b>0$''. Replace it with ``Let $b>1$''.

\item
On page~9, in the proof of Corollary~1, replace
$\mu_{k-k(n);m+1}$ with $u_{k-k(n);m+1}$.

\item
On page~15, in Remark~9, one reads that
``$\mu_k([x,x+b^{-l}))$ is $0$ for $k<j$ and $b^{1-l}$ for
$j\geq k$''. The latter should of course read ``for $k\geq j$''.

\item
In the proof of Theorem~3 at the bottom of page 16, it would have been
better to write
$\lim_{k\to\infty}\int_{[0,1)}\frac{\mu_{k-k(n)}(dx)}{n+x}$. The shift
``$-k(n)$'' from Proposition~4 was deliberately omitted, as it does not
affect the limit.

\end{itemize}

\paragraph{Acknowledgement.} The author is very grateful to Jeffrey Shallit
for bringing the issue with Lemma~3 to his attention.

\end{document}